\documentclass[a4paper]{amsart}
\usepackage{amssymb}
\usepackage{xypic}
\usepackage[only,mapsfrom]{stmaryrd}
\usepackage{graphicx}

\newtheorem{theorem}{Theorem}[section]
\newtheorem{lemma}[theorem]{Lemma}

\newtheorem{cor}[theorem]{Corollary}

\theoremstyle{definition}

\theoremstyle{remark}
\newtheorem{remark}[theorem]{Remark}

\numberwithin{equation}{section}

\newcommand{\C}{\mathbb{C}}

\newcommand{\R}{\mathbb{R}}

\title[Classification of quasitoric manifolds]{Remarks on the classification of quasitoric manifolds up to equivariant homeomorphism}

\author{Michael Wiemeler}
\address{School of Mathematics,
    Alan Turing Building, The University of Manchester, Oxford Road, Manchester M13 9PL, UK}
\email{michael.wiemeler-2@manchester.ac.uk}
\thanks{Part of the research was supported by SNF Grants Nos. 200021-117701, 200020-126795 and PBFR2-133466.}



\begin{document}
\begin{abstract}
  We give three sufficient criteria for two quasitoric manifolds \(M,M'\) to be (weakly) equivariantly homeomorphic.
  We apply these criteria to count the weakly equivariant homeomorphism types of quasitoric manifolds with a given cohomology ring.
\end{abstract}

\maketitle

%-----------------------------------------------------------------------
% End of amsart.template
%-----------------------------------------------------------------------

\section{Introduction}

In topology one studies invariants of topological spaces.
Examples of such invariants are the cohomology ring of a space or the bordism type of a  manifold.
If two spaces have different invariants one knows that these spaces are not homeomorphic.
There is also the inverse problem:
Are two spaces having the same invariants homeomorphic?

In this note we study certain invariants of so called quasitoric manifolds and give a positive answer to the above question for these invariants.

Quasitoric manifolds were introduced by Davis and Januszkiewicz \cite{davis91:_convex_coxet}.
A quasitoric manifold is a \(2n\)-dimensional manifold on which an \(n\)-dimensional torus acts such that the orbit space is a simple polytope.

We give three sufficient criteria for two quasitoric manifolds \(M,M'\) to be (weakly) equivariantly homeomorphic.
The first criterion gives a condition on the cohomology of \(M\) and \(M'\) (see section~\ref{sec:iso}).
We apply this criterion to count the weakly equivariant homeomorphism types of quasitoric manifolds with the same cohomology ring as \(\C P^n\#\C P^n\), \(\C P^n\#\bar{\C P}^n\) with \(n > 2\). Moreover, we show that a quasitoric manifold with the same cohomology as \(\prod_{i=1}^l\C P^{n_i}\), \(n_i>1\) for all \(i\), is weakly equivariantly homeomorphic to \(\prod_{i=1}^l\C P^{n_i}\).

The stable tangent bundle of a quasitoric manifold \(M\) splits as a sum of complex line bundles.
This induces a \(BT^m\)-structure on the stable tangent bundle of \(M\).
We show in section~\ref{sec:bordism} that two \(BT^m\)-bordant quasitoric manifolds are weakly equivariantly homeomorphic.

In section~\ref{sec:gkm} we show that two quasitoric manifolds having the same GKM-graphs are equivariantly homeomorphic.

In this paper we take all cohomology groups with coefficients in \(\mathbb{Z}\).

The results presented in this note form part of the outcome of my Ph.D. thesis \cite{wiemeler_phd} written under the supervision of Prof. Anand Dessai at the University of Fribourg.
I would like to thank Anand Dessai for helpful discussions.
I would also like to thank Nigel Ray for comments on an earlier version of this paper.

\section{Isomorphisms of cohomology rings}
\label{sec:iso}

At first we introduce some notations concerning quasitoric manifolds and their characteristic functions.
We follow \cite{semifree} for this description.
A quasitoric manifold is a \(2n\)-dimensional manifold with a locally standard action of an \(n\)-dimensional torus \(T\) such that the orbit space \(M/T\) is face preserving homeomorphic to an \(n\)-dimensional simple polytope \(P\).
We denote the orbit map by \(\pi:M\rightarrow P\).
Furthermore we denote the set of facets of \(P\) by \(\mathfrak{F}=\{F_1,\dots,F_m\}\).
The \emph{characteristic submanifolds} \(M_i=\pi^{-1}(F_i)\), \(i=1,\dots,m\), of \(M\) are the preimages of the facets of \(P\).
Each \(M_i\) is fixed pointwise by a one-dimensional subtorus \(\lambda(F_i)=\lambda(M_i)\) of \(T\). 

Recall that two simple polytopes are called \emph{combinatorially equivalent} if there is an inclusion preserving bijection between their face lattices.
It is said that two simple polytopes have the same combinatorial type if they are combinatorially equivalent.

The following lemma was proved by Davis and Januszkiewicz~\cite[p. 424]{davis91:_convex_coxet}:
\begin{lemma}
  A quasitoric manifold \(M\) with \(P=M/T\) is determined up to equivariant homeomorphism by the combinatorial type of \(P\) and the function \(\lambda\).
\end{lemma}

Let \(N\) be the integer lattice of one-parameter circle subgroups in \(T\), so we have \(N\cong \mathbb{Z}^n\).
We denote by  \(\bar{\lambda}: \mathfrak{F}\rightarrow N\) the characteristic function of \(M\).
Then, for a given  facet \(F_i\) of \(P\), \(\bar{\lambda}(F_i)\) is a primitive vector that spans \(\lambda(F_i)\).
The vector \(\bar{\lambda}(F_i)\) is determined up to sign by this condition.

An \emph{omniorientation} of \(M\) is a choice of orientations for each \(M_i\) and \(M\).
It helps to eliminate the indeterminateness in the definition of a characteristic function.
This is done as follows:
An omniorientation of \(M\) determines orientations for all normal bundles of the characteristic submanifolds of \(M\).
The action of a one-parameter circle subgroup of \(T\) also determines orientations for these bundles.
We choose the primitive vectors \(\bar{\lambda}(F_i)\) in such a way that the two orientations on \(N(M_i,M)\) coincide.

A characteristic function satisfies the following non-singularity condition.
For pairwise distinct facets \(F_{j_1},\dots,F_{j_n}\) of \(P\),
\begin{equation*}
  \bar{\lambda}(F_{j_1}),\dots,\bar{\lambda}(F_{j_n})
\end{equation*}
forms a basis of \(N\) whenever the intersection
\begin{equation*}
  F_{j_1}\cap \dots \cap F_{j_n}
\end{equation*}
is non-empty.
After reordering the facets we may assume that
\begin{equation*}
  F_{1}\cap \dots \cap F_{n}\neq \emptyset.
\end{equation*}
Therefore   \(\bar{\lambda}(F_{1}), \dots ,\bar{\lambda}( F_{n})\) is a basis of \(N\).
This allows us to identify \(N\) with \(\mathbb{Z}^n\).
This identification induces an identification of the torus \(T\) with the standard \(n\)-dimensional torus \(\mathbb{R}^n/\mathbb{Z}^n\).

With this identifications understood, we may write \(\bar{\lambda}\) as an integer matrix of the form
\begin{equation}
\label{eq:matrix}
\Lambda=
  \begin{pmatrix}
    1 & & & \lambda_{1,n+1}& \dots & \lambda_{1,m}\\
      &  \ddots& &\vdots& &\vdots\\
      & &1&\lambda_{n,n+1}&\dots&\lambda_{n,m}
  \end{pmatrix}.
\end{equation}
With this notation \(\lambda(F_i)\), \(i=1,\dots,m\) is given by
\begin{equation*}
  \left\{t
    \begin{pmatrix}
      \lambda_{1,i}\\
      \vdots\\
      \lambda_{n,i}
    \end{pmatrix}
\in \mathbb{R}^n/\mathbb{Z}^n; t\in \mathbb{R}
\right\}.
\end{equation*}

Let \(u_i\in H^2(M)\) be the Poincar\'e-dual of the characteristic submanifold \(M_i\). Then the cohomology ring \(H^*(M)\) is generated by \(u_1,\dots,u_m\).
The \(u_i\) are subject to the following relations \cite[p. 439]{davis91:_convex_coxet}:
\begin{enumerate}
\item \(\forall I\subset \{1,\dots,m\}\) \(\prod_{i\in I}u_i=0\Leftrightarrow \bigcap_{i\in I}F_i =\emptyset\)
\item For \(i=1,\dots,n\)  \(-u_i=\sum_{j=n+1}^m\lambda_{i,j}u_j\).
\end{enumerate}

Two quasitoric manifolds \(M,M'\) are \emph{weakly \(T\)-equivariantly homeomorphic} if there is an automorphism \(\theta: T\rightarrow T\) and a homeomorphism \(f:M\rightarrow M'\) such that for all \(x\in M\) and \(t\in T\):
\begin{equation*}
  f(tx)=\theta(t)f(x).
\end{equation*}
Because the identification of \(T\) with \(\R^n/\mathbb{Z}^n\) depends on a choice of a basis in \(N\) a quasitoric manifold \(M\) is determined by the combinatorial type of \(P\) and the characteristic matrix \(\Lambda\) only up to weakly equivariant homeomorphism.

Now we are in the position to prove our first theorem.
\begin{theorem}
\label{sec:isom-cohom-rings}
  Let \(M,M'\) be quasitoric manifolds of dimension \(n\).
  Furthermore let \(u_1,\dots,u_m\in H^2(M)\) be the Poincar\'e-duals of the characteristic submanifolds of \(M\) and  \(u_1',\dots,u_{m'}'\in H^2(M')\) the Poincar\'e-duals of the characteristic submanifolds of \(M'\).
  If there is a ring isomorphism \(f:H^*(M)\rightarrow H^*(M')\) and a permutation \(\sigma:\{1,\dots,m'\}\rightarrow \{1,\dots,m'\}\) with \(f(u_i)=\pm u_{\sigma(i)}'\), \(i=1,\dots,m\), then \(M\) and \(M'\) are weakly T-equivariantly homeomorphic.
\end{theorem}
\begin{proof}
  After reordering the \(u_i'\), we may assume that \(f(u_i)=\pm u_i'\) for \(i=1,\dots,m\).
  After changing the orientations of the characteristic submanifolds of \(M'\), we may assume that \(f(u_i)=u_i'\).

  At first notice that \(f\) preserves the grading of \(H^*(M)\) and
  \begin{equation*}
    m=b_2(M)+n=b_2(M')+n=m'.
  \end{equation*}
For \(I\subset \{1,\dots,m\}\), we have
\begin{align*}
  &\bigcap_{i\in I}F_i =\emptyset\\
  \Leftrightarrow& \prod_{i\in I} u_i =0\\
  \Leftrightarrow& \prod_{i\in I} u_i'=\prod_{i\in I}f( u_i) =0\\
  \Leftrightarrow&\bigcap_{i\in I}F_i' =\emptyset.
\end{align*}
Here \(F_i,F_i'\) denote the facets of \(M/T\) and \(M'/T\), respectively.
Therefore \(M/T\) and \(M'/T\) are combinatorially equivalent.

Now we show that the characteristic matrices of \(M\) and \(M'\) are equal.
We may assume that \(F_1\cap\dots\cap F_n\neq \emptyset\neq F_1'\cap\dots\cap F_n'\).
Then \(u_{n+1},\dots, u_m\) forms a basis of \(H^2(M)\) and \(u_{n+1}',\dots, u_m'\) a basis of \(H^2(M')\).

If we write the characteristic matrices \(\Lambda,\Lambda'\) for \(M,M'\) in the form~(\ref{eq:matrix}) then we have
\begin{align*}
  -u_i =& \sum_{j=n+1}^m \lambda_{i,j}u_j\\
 -u_i' =& \sum_{j=n+1}^m \lambda_{i,j}'u_j'
\end{align*}
for \(i=1,\dots,n\).
Therefore we have
\begin{equation*}
   \sum_{j=n+1}^m \lambda_{i,j}'u_j'= -u_i' =  f( -u_i) = \sum_{j=n+1}^m \lambda_{i,j}f(u_j)=  \sum_{j=n+1}^m \lambda_{i,j}u_j'.
\end{equation*}
It follows that \(\lambda_{i,j}'=\lambda_{i,j}\), \(i=1,\dots,n\), \(j=n+1,\dots,m\).
Therefore the characteristic matrices are the same.
\end{proof}

\begin{remark}
\label{sec:isom-cohom-rings-1}
  If \(M,M'\) are two weakly \(T\)-equivariant homeomorphic quasitoric manifolds, then there is an isomorphism \(H^*(M)\rightarrow H^*(M')\) with the properties described in Theorem~\ref{sec:isom-cohom-rings} because a weakly \(T\)-equivariant homeomorphism \(M\rightarrow M'\) maps the characteristic submanifolds of \(M\) homeomorphicly on the characteristic submanifolds of \(M'\).
\end{remark}

\begin{cor}
  Let \(\alpha(n)\), \(\bar{\alpha}(n)\) be the number of quasitoric manifolds up to weakly equivariant homeomorphism which have the same cohomology ring as \(\C P^n\#\C P^n\) or \(\C P^n\# \bar{\C P}^n\), respectively.
If \(n \geq 3\), we have
\begin{equation*}
  \alpha(n)=
  \begin{cases}
    \frac{n+1}{2}& \text{if } n \equiv 1 \mod 2\\
    \frac{n}{4}& \text{if } n \equiv 0 \mod 4\\
    \frac{n+2}{4}&\text{if } n \equiv 2 \mod 4,
  \end{cases}
\end{equation*}
\begin{equation*}
  \bar{\alpha}(n)=
  \begin{cases}
    \frac{n+1}{2}& \text{if } n \equiv 1 \mod 2\\
    \frac{n}{4}+1& \text{if } n \equiv 0 \mod 4\\
    \frac{n+2}{4}&\text{if } n \equiv 2 \mod 4.
  \end{cases}
\end{equation*}
\end{cor}
\begin{proof}
  Let \(M\) be a quasitoric manifold with
  \begin{equation*}
    H^*(M)=H^*(\C P^n \#\C P^n)=\mathbb{Z}[a,b]/(ab, a^n-b^n).   
  \end{equation*}
Then, by Theorem 5.3 of \cite[p. 353]{pre05797869}, the orbit polytope of \(M\) is \([0,1]\times \Delta^n\), where \(\Delta^n\) is the \(n\)-dimensional simplex.

Let \(u_1,u_2\in H^2(M)\) be the Poincar\'e-duals of \(\pi^{-1}(\{0\}\times \Delta^n)\) and \(\pi^{-1}(\{1\}\times \Delta^n)\), respectively.
Moreover, let \(u_3,\dots,u_{n+2}\in H^*(M)\) be the Poincar\'e-duals of the other characteristic submanifolds of \(M\).
Let \(\alpha_1,\dots,\alpha_{n+2}\in \mathbb{Z}\) and \(\beta_1,\dots,\beta_{n+2}\in \mathbb{Z}\) such that
\begin{equation*}
  u_i=\alpha_i a + \beta_i b.
\end{equation*}
Because \(0=u_1u_2=\alpha_1\alpha_2 a^2 + \beta_1\beta_2 b^2\) and \(n\geq 3\), we may assume that \(\alpha_2=\beta_1=0\).
Because
\begin{equation*}
  \alpha_i\prod_{k=1}^{n-1}\alpha_{j_k} + \beta_i\prod_{k=1}^{n-1}\beta_{j_k}=\langle u_i u_{j_1}\dots u_{j_{n-1}},[M]\rangle = \pm 1,
\end{equation*}
for \(i\in \{1,2\}\) and \(j_k\in \{3,\dots,n+2\}\), \(j_{k_1}\neq j_{k_2}\) for \(k_1\neq k_2\), we see that \(|\alpha_j|=|\beta_j|=1\) for \(j\geq 3\) and \(\alpha_1=\pm 1\), \(\beta_2=\pm 1\).

After changing the orientations of the characteristic submanifolds of \(M\) we may assume that
\begin{align*}
  \alpha_1&=\beta_2=1 & \alpha_j&=1 \text{ for } j\geq 3.
\end{align*}
Let \(k_M=\#\{j\in \{3,\dots,n+2\};\;\beta_j=-1\}\).
Then we have
\begin{equation*}
  a^n + (-1)^{k_M}b^n =u_3\dots u_{n+2}=0.
\end{equation*}
Therefore \(0\leq k_M\leq n\) is odd.

\emph{Claim.} Let \(M'\) be another quasitoric manifold with \(H^*(M')=H^*(\C P^n \# \C P^n)\). Then \(M\) and \(M'\) are weakly \(T\)-equivariantly homeomorphic if and only if \(k_M=k_{M'}\) if \(n\) is odd or \(k_M=k_{M'}\) or \(k_M=n-k_{M'}\) if \(n\) is even.

By Theorem~\ref{sec:isom-cohom-rings} and Remark~\ref{sec:isom-cohom-rings-1}, \(M\) and \(M'\) are weakly \(T\)-equivariantly homeomorphic if and only if there is an automorphism \(f\) of \(H^*(\C P^n\# \C P^n)\) and a permutation \(\sigma:\{1,\dots,n+2\}\rightarrow\{1,\dots,n+2\}\) such that \(f(u_i)= \pm u_{\sigma(i)}'\).

If \(k_M=k_{M'}\), we may take \(f\) to be the identity and \(\sigma\) a suitable permutation to see that \(M\) and \(M'\) are weakly equivariantly homeomorphic.

Now assume that \(M,M'\) are weakly \(T\)-equivariantly homeomorphic.
We first discuss the case where \(n\) is odd.
Then the automorphism group of \(H^*(\C P^n\#\C P^n)\) is generated by two automorphisms \(f_1,f_2\) with 
\begin{align*}
  f_1(a)&=b&f_1(b)&=a\\
  f_2(a)&=-a&f_2(b)&=-b.
\end{align*}
Because
\begin{equation*}
  f_1(u_i)=
  \begin{cases}
    u_2&\text{if } i=1\\
    u_1&\text{if } i=2\\
    -u_i&\text{if } i\geq 3 \text{ and } \beta_i =-1\\
    u_i&\text{if } i\geq 3 \text{ and } \beta_i =1
  \end{cases}
\end{equation*}
and \(f_2(u_i)=-u_i\), we see that \(k_M=k_{M'}\).

Now we turn to the case where \(n\) is even.
Then the automorphism group of \(H^*(\C P^n\#\C P^n)\) is generated by two automorphisms \(f_1,f_2\) with
\begin{align*}
  f_1(a)&=b&f_1(b)&=a\\
  f_2(a)&=a&f_2(b)&=-b.
\end{align*}
Because
\begin{equation*}
  f_1(u_i)=
  \begin{cases}
    u_2&\text{if } i=1\\
    u_1&\text{if } i=2\\
    -u_i&\text{if } i\geq 3 \text{ and } \beta_i =-1\\
    u_i&\text{if } i\geq 3 \text{ and } \beta_i =1
  \end{cases}
\end{equation*}
and 
\begin{equation*}
  f_2(u_i)=
  \begin{cases}
    u_1&\text{if } i=1\\
    -u_2&\text{if } i=2\\
    a+b&\text{if } i\geq 3 \text{ and } \beta_i =-1\\
    a-b&\text{if } i\geq 3 \text{ and } \beta_i =1,
  \end{cases}
\end{equation*}
 we see that \(k_M=k_{M'}\) or \(k_M=n-k_{M'}\).

It follows from our claim that
\begin{equation*}
  \alpha(n)\leq
  \begin{cases}
    \frac{n+1}{2}& \text{if } n \equiv 1 \mod 2\\
    \frac{n}{4}& \text{if } n \equiv 0 \mod 4\\
    \frac{n+2}{4}&\text{if } n \equiv 2 \mod 4.
  \end{cases}
\end{equation*}

The only thing which remains to be proven is that for each \(0\leq k\leq n\) with \(k\) odd there is a quasitoric manifold \(M\) with \(k_M=k\) and \(H^*(M)=H^*(\C P^n\#\C P^n)\).

As we constructed a characteristic function for a quasitoric manifold, we may also construct a quasitoric manifold from a polytope \(P\) and a function \(\bar{\lambda}:\mathfrak{F}\rightarrow N\cong \mathbb{Z}^n\) satisfying a non-singularity condition \cite[p. 423]{davis91:_convex_coxet}.

Let \(M\) be the quasitoric manifold over \([0,1]\times \Delta^n\) defined by the characteristic function
\begin{align*}
  \bar{\lambda}(\{0\}\times \Delta^n)&=(1,\dots,1)^t\\
\bar{\lambda}(\{1\}\times \Delta^n)&=(-1,\dots,-1,1,\dots,1)^t\\
\bar{\lambda}([0,1]\times F_i)&=(0,\dots,0,-1,0,\dots,0)^t,
\end{align*}
where \(k\) entries of \(\bar{\lambda}(\{1\}\times \Delta^n)\) are equal to \(-1\), the \(F_i\) are the facets of \(\Delta^n\) and the \(i\)th entry of \(\bar{\lambda}([0,1]\times F_i)\) is equal to \(-1\).

Then we have
\begin{align*}
  H^*(M)&=\mathbb{Z}[a,b]/(ab, (a-b)^k(a+b)^{n-k})\\
&= \mathbb{Z}[a,b]/(ab, a^n -b^n)\\
&=H^*(\C P^n\# \C P^n)
\end{align*}
and \(k_M=k\).

The case where \(H^*(M)=H^*(\C P^n \# \bar{\C P}^n)\) is similar.
The only difference in this case is that \(0\leq k_M\leq n\) is even in this case.
We omit the details.
\end{proof}

\begin{cor}
\label{sec:isom-cohom-rings-2}
  Let \(M\) be a quasitoric manifold with the same cohomology as \(\prod_{i=1}^l \C P^{n_i}\), \(n_i>1\) for all \(i=1,\dots,l\).
Then \(M\) is weakly equivariantly homeomorphic to \(\prod_{i=1}^l \C P^{n_i}\).
\end{cor}
\begin{proof}
  We have
  \begin{equation*}
    H^*(M)=\mathbb{Z}[x_1,\dots,x_l]/(x_i^{n_i+1};\;i=1,\dots,l)
  \end{equation*}
with \(\deg x_i =2\).
By Theorem 5.3 of \cite[p. 353]{pre05797869}, we know that the orbit polytope of \(M\) is a product \(\prod_{i=1}^l\Delta^{n_i}\) of simplices.

For \(i=1,\dots,l\), \(j=1,\dots,n_i+1\), let \(u_{ij}\in H^2(M)\) be the Poincar\'e-dual of the characteristic submanifold 
\begin{equation*}
\pi^{-1}(\Delta^{n_1}\times\dots\times \Delta^{n_{i-1}}\times F_j\times \Delta^{n_{i+1}}\times\dots\times \Delta^{n_l}),  
\end{equation*}
where \(F_j\) is a facet of \(\Delta^{n_i}\).
Let \(\alpha_{ijk}\in \mathbb{Z}\) such that \(u_{ij}=\sum_{k=1}^l \alpha_{ijk}x_k\).

Because \(\prod_{j=1}^{n_i+1}u_{ij}=0\), we have
\begin{equation}
  \label{eq:prod}
  \prod_{j=1}^{n_i+1}\left(\sum_{k=1}^l \alpha_{ijk} x_k\right)=\sum_{k=1}^l f_{ik}x_k^{n_k+1}
\end{equation}
in \(\mathbb{Z}[x_1,\dots,x_l]\), where \(f_{ik}\in \mathbb{Z}[x_1,\dots,x_l]\) such that \(\deg f_{ik}\leq 2(n_i+1)-2(n_k+1)\).

By replacing those \(x_k\) by \(0\) in (\ref{eq:prod}) for which \(n_k<n_i\), we get
\begin{equation}
\label{eq:1}
  \prod_{j=1}^{n_i+1}\left(\sum_{k,\; n_k\geq n_i} \alpha_{ijk} x_k\right)=\sum_{k;\; n_k=n_i} f_{ik}x_k^{n_k+1}.
\end{equation}
Because the \(f_{ik}\) in this equation are integers, it follows that \(\alpha_{ijk}=0\) whenever \(n_k>n_i\).

Now assume that there are \(k_1,k_2\in \{1,\dots,l\}\), \(k_1\neq k_2\) such that \(f_{ik_1}\neq 0 \neq f_{ik_2}\) and \(n_{k_1}=n_{k_2}=n_i\). 
Then by replacing \(x_k\), for \(k\neq k_1,k_2\), by \(0\) and \(x_{k_2}\) by \(1\) in (\ref{eq:prod}) we get
\begin{equation*}
    \prod_{j=1}^{n_i+1}\left(\alpha_{ijk_1} x_{k_1}+\alpha_{ijk_2}\right)= f_{ik_1}x_k^{n_k+1} + f_{ik_2}.
\end{equation*}
But because \(n_k>1\), not all roots of the polynomial on the right hand side of this equation are rational.
This is a contradiction. 
Therefore for each \(i\in\{1,\dots,l\}\) there is a unique \(\sigma(i)\in\{1,\dots,l\}\) such that \(n_i=n_{\sigma(i)}\) and \(f_{ik}=0\) if \(k\neq \sigma(i)\) and \(n_k\geq n_i\).

Because \(\mathbb{Z}[x_1,\dots,x_l]\) is a factorial ring, we see by equation (\ref{eq:1}) that \(\alpha_{ijk}=0\), for all \(j\), if \(k\neq \sigma(i)\) and \(n_k\geq n_i\).
Because each tuple \((u_{1j_1},\dots,u_{lj_l})\) forms a basis of \(H^2(M)\), we see that \(\sigma\) is bijective and \(\alpha_{ij\sigma(i)}=\pm 1\).

After changing the omniorientation of \(M\), we may assume that all \(\alpha_{ij\sigma(i)}=1\).
Let \(k_1\in \{1,\dots,l\}\) such that \(n_{k_1}<n_i\). Then by replacing those \(x_k\) by \(0\) for which \(k\neq k_1,\sigma(i)\) we get from (\ref{eq:prod}):
\begin{equation*}
  \prod_{j=1}^{n_i+1}(x_{\sigma(i)} + \alpha_{ijk_1}x_{k_1})=f_{i\sigma(i)}x_{\sigma(i)}^{n_i+1} + f_{ik_1}x_{k_1}^{n_{k_1}+1}.
\end{equation*}
Because \(n_{k_1}>1\), we get by comparing coefficients that
\begin{gather*}
  \left(\sum_{1\leq j_1< j_2\leq n_i+1} \alpha_{ij_1k_1}\alpha_{ij_2k_1}\right) x_{\sigma(i)}^{n_i-1}x_{k_1}^2 =0,\\
  \left(\sum_{j=1}^{n_i+1}\alpha_{ijk_1}\right) x_{\sigma(i)}^{n_i}x_{k_1}=0.
\end{gather*}
Because
\begin{equation*}
  \sum_{j=1}^{n_i+1}\alpha_{ijk_1}^2 =  \left(\sum_{j=1}^{n_i+1}\alpha_{ijk_1}\right)^2 -2\sum_{1\leq j_1< j_2\leq n_i+1} \alpha_{ij_1k_1}\alpha_{ij_2k_1},
\end{equation*}
we see that \(\alpha_{ijk_1}=0\) for all \(j\).
Therefore we have \(u_{ij}=x_{\sigma(i)}\).
 Hence, the statement follows from Theorem \ref{sec:isom-cohom-rings}.
\end{proof}

\begin{remark}
  It follows from results of \cite{pre05797869} and \cite{1195.57060} that a quasitoric manifold with the same cohomology as a product of complex projective spaces is homeomorphic to a product of complex projective spaces.

The assumption in Corollary~\ref{sec:isom-cohom-rings-2} that \(n_i>1\) is necessary because there are infinitely many non-equivalent torus actions on \(\C P^1\times \C P^1\) \cite{0486.57016}.
\end{remark}

\section{Bordism}
\label{sec:bordism}

To state our second theorem we first fix some notation.
Let \(M\) be a  omnioriented quasitoric manifold.
By \cite[p. 446]{davis91:_convex_coxet} and \cite[p. 71]{1012.52021} there is an isomorphism of real vector bundles
\begin{equation*}
  TM\oplus \R^{2(m-n)}\cong L_1\oplus \dots\oplus L_m,
\end{equation*}
where the \(L_i\) are complex line bundles with
\begin{equation*}
  c_1(L_i)=u_i.
\end{equation*}
This isomorphism corresponds to a reduction  of structure group in the stable tangent bundle of \(M\) from \(O(2m)\) to \(T^m\).

Let \(g: M\rightarrow BO(2m)\) be  a classifying map for the stable tangent bundle of \(M\).
Furthermore let \(f_i:M\rightarrow BT^1\) be the classifying map of the line bundle \(L_i\).
Then the following diagram commutes up to homotopy:
\begin{equation*}
  \xymatrix{
    & BT^m \ar@{=}[r]\ar^{p_m}[d] & BT^1\times\dots\times BT^1\\
    M\ar^{\prod f_i}[ru]\ar[r]&BO(2m)}
\end{equation*}
where \(p_m\) is the natural fibration \cite[p. 77]{0794.55001}.
We may replace \(\prod f_i\) by a homotopic map \(f\) which makes the above diagram commutative.
By \(f\) there is  given a \((BT^m,p_m)\)-structure on the stable tangent bundle of \(M\) \cite[p. 14]{0181.26604}.
We denote by \(\Omega_n(BT^\infty,p)\) the bordism groups of the sequence
\begin{equation*}
  \xymatrix{
    \dots \ar[r]& BT^m \ar[rr]\ar^{p_m}[d] & & BT^{m+1}\ar[r]\ar^{p_{m+1}}[d]& \dots\\
    \dots \ar[r]& BO(2m) \ar[r] & BO(2m+1)\ar[r] & BO(2m+2)\ar[r]& \dots}
\end{equation*}

\begin{theorem}
\label{sec:bordism-1}
  Let \(M,M'\) be omnioriented quasitoric manifolds with \([M]=[M']\in \Omega_n(BT^\infty,p)\).
  Then \(M\) and \(M'\) are weakly \(T\)-equivariantly homeomorphic.
\end{theorem}
\begin{proof}
  We use the following notation.
  Let \(f:M\rightarrow BT^\infty, L_1,\dots,L_m\) as above and \(f':M'\rightarrow BT^\infty, L'_1,\dots,L'_{m'}\) analogous.
  Let \(\{F_1,\dots,F_m\}\) and  \(\{F'_1,\dots,F'_{m'}\}\) be the set of facets of \(M/T=P\) and \(M'/T=P'\), respectively.
  
  Furthermore let 
  \begin{equation*}
    H^*(BT^\infty)= \mathbb{Z}[x_1,x_2,x_3,\dots].
  \end{equation*}
  Then be have
  \begin{equation}
    \label{eq:bor1}
    f^*(x_i)=
    \begin{cases}
      c_1(L_i)&\text{if }i=1,\dots,m\\
      0 & \text{else}.
    \end{cases}
  \end{equation}
  Without loss of generality we may assume that \(m'\geq m\).
  Because bordant manifolds have the same characteristic numbers, for all \(i_1,\dots,i_n\in \{1,\dots,m'\}\) we get
  \begin{equation*}
    f^*(x_{i_1}\dots x_{i_n})[M]= f'^*(x_{i_1}\dots x_{i_n})[M'].
  \end{equation*}
  If the \(i_j\) are pairwise distinct then we have by~(\ref{eq:bor1})
  \begin{equation*}
     f^*(x_{i_1}\dots x_{i_n})[M]=
     \begin{cases}
       \pm 1 &\text{if } i_j \leq m \text{ and } F_{i_1}\cap\dots\cap F_{i_n}\neq \emptyset\\
       0 & \text{else}.
     \end{cases}
  \end{equation*}
Since this holds analogously for \(M'\) we get
\begin{align}
  m&=m',\\
  \label{eq:bor2}
  F_{i_1}\cap\dots\cap F_{i_n}=\emptyset &\Leftrightarrow F'_{i_1}\cap\dots\cap F'_{i_n}=\emptyset.
\end{align}
By~(\ref{eq:bor2}) \(P\) and \(P'\) are combinatorially equivalent.
An equivalence is given by
\begin{equation*}
  \bigcap_{i\in I}F_i\mapsto \bigcap_{i\in I}F'_i,\quad (I\subset \{1,\dots,m\}).
\end{equation*}
Without loss of generality we may assume that \(F_1\cap\dots\cap F_n\) is non-empty.
Then \(f^*(x_{n+1}),\dots,f^*(x_m)\) form a basis of \(H^2(M)\).
Similarly  \(f'^*(x_{n+1}),\dots,f'^*(x_m)\) form a basis of \(H^2(M')\).
Therefore there is an isomorphism
\begin{align*}
  \psi: H^2(M)&\rightarrow H^2(M')\\
  f^*(x_i)&\mapsto f'^*(x_i),\quad i>n.
\end{align*}

We claim that the following diagram commutes.
\begin{equation}
\label{eq:bor3}
  \xymatrix{
    H^{2n-2}(M)\ar^\cong[r]&\hom(H^2(M),\mathbb{Z})\\
    H^{2n-2}(BT^\infty)\ar@{>>}^{f^*}[u]\ar@{>>}_{f'^*}[d]\\
    H^{2n-2}(M')\ar^\cong_{x\mapsto \langle \cdot \cup x,[M']\rangle}[r]&\hom(H^2(M'),\mathbb{Z})\ar^\cong_{\psi^*}[uu]}
\end{equation}
Let \(x\in H^{2n-2}(BT^\infty)\). Then for \(i>n\) we have
\begin{align*}
  \psi^*(\langle \cdot \cup f'^*(x),[M']\rangle)(f^*(x_i))&=\langle f'^*(x_i) \cup f'^*(x),[M']\rangle\\
  &=\langle f^*(x_i) \cup f^*(x),[M]\rangle &&\text{by bordism}\\
  &=(\langle \cdot \cup f^*(x),[M]\rangle)(f^*(x_i)).
\end{align*}
Therefore the diagram commutes.
Now we have for \(i=1,\dots,n\) and \(x\in H^{2n-2}(BT^\infty)\):
\begin{align*}
  \langle \psi(f^*(x_i)) \cup f'^*(x),[M']\rangle &=\psi^*(\langle \cdot \cup f'^*(x),[M']\rangle)(f^*(x_i))\\
&=(\langle \cdot \cup f^*(x),[M]\rangle)(f^*(x_i))&&\text{by~(\ref{eq:bor3})}\\
&=\langle f^*(x_i) \cup f^*(x),[M]\rangle\\
&=\langle f'^*(x_i) \cup f'^*(x),[M']\rangle &&\text{by bordism}
\end{align*}
Because \(f'^*:H^{2n-2}(BT^\infty)\rightarrow H^{2n-2}(M')\) is surjective, it follows that
\begin{equation*}
  f'^*(x_i)=\psi(f^*(x_i)), \quad \text{for } i=1,\dots,n.
\end{equation*}
As in the proof of Theorem~\ref{sec:isom-cohom-rings} one sees that the characteristic matrices for \(M\) and \(M'\) are equal.
Therefore \(M\) and \(M'\) are weakly equivariantly homeomorphic.
\end{proof}

The map \(p_m:BT^m \rightarrow BO(2m)\) factors through \(BU(m)\).
Therefore an omniorientation of a quasitoric manifold \(M\) induces a stable almost complex structure on \(M\).

If, moreover, \(M=\C P^2\#\C P^2\) we may choose this stable almost complex structure in such a way that \(\langle c_2(M),[M]\rangle=2\).

\begin{cor}
\label{sec:bordism-2}
  Let \(M\) be a four-dimensional quasitoric manifold.
  Equip \(M\) and \(\C P^2\#\C P^2\) with the stable almost complex structures described above.
  If \(\chi(M)\leq 4\) and \([M]=[\C P^2\# \C P^2]\in \Omega(BU)\), then \(M\) is weakly \(T\)-equivariantly homeomorphic to \(\C P^2\#\C P^2\).
\end{cor}
\begin{proof}
  Because \(M\) is bordant to \(\C P^2\#\C P^2\), it has signature two.
  Therefore we have \(b_2(M)\geq 2\). 
  This implies that \(\chi(M)=4\) and that the intersection form of \(M\) is positive definite.
  Hence, the orbit polytope \(P\) of \(M\) is a square.
  Denote by \(F_1,\dots,F_4\) the facets of \(P\) such that \(F_i\cap F_j=\emptyset\) if and only if \(i-j\equiv 0 \mod 2\) and \(i\neq j\).
  Let \(u_i\in H^2(M)\), \(i=1,\dots,4\) be the Poincar\'e-duals, of \(\pi^{-1}(F_i)\).
Then we have
\begin{gather*}
  \sum_{i=1}^4 \langle u_i^2,[M]\rangle= \langle p_1(M),[M]\rangle =6,\\
  \sum_{1\leq i<j\leq 4} \langle u_iu_j,[M]\rangle = \langle c_2(M),[M]\rangle=2,\\
  \langle u_iu_j,[M]\rangle=
  \begin{cases}
    \pm 1& \text{if } i-j\equiv 1 \mod 2\\
    0&\text{if } i-j\equiv 0\mod 2 \text{ and } i\neq j,
  \end{cases}\\
 \langle u_i^2,[M]\rangle \geq 1 \text{ for } i=1,\dots,4.
\end{gather*}

Therefore we must have up to ordering
\begin{align*}
  (\langle u_1^2,[M]\rangle,\dots,\langle u_4^2,[M]\rangle)&=(1,1,2,2),(1,1,1,3),\\
  (\langle u_iu_j,[M]\rangle ;\; i-j\equiv 1 \mod 2)&=(-1,1,1,1).
\end{align*}

Assume that there are \(i,j\in \{1,\dots,4\}\) such that \(i-j\equiv 1\mod 2\) and
\begin{equation*}
  \langle u_i^2,[M]\rangle=\langle u_j^2,[M]\rangle =1.
\end{equation*}
Then \(u_i,u_j\) form a basis of \(H^2(M)\).
In particular, for \(k\in \{1,\dots,4\}-\{i,j\}\), there are \(\alpha,\beta \in \mathbb{Z}\) such that \(u_k=\alpha u_i + \beta u_j\).
 It follows that
\begin{align*}
  \langle u_k^2,[M]\rangle &= \alpha^2\langle u_i^2,[M]\rangle + \beta^2\langle u_j^2,[M]\rangle + 2\alpha\beta \langle u_i u_j, [M]\rangle\\
&=(\alpha \pm \beta)^2.
\end{align*}
This is a contradiction, because there is a \(k\) with \(\langle u_k^2,[M]\rangle = 2,3\).

Therefore we must have \(i-j\equiv 0\mod 2\) whenever
\begin{equation*}
  \langle u_i^2,[M]\rangle=\langle u_j^2,[M]\rangle =1.
\end{equation*}
This implies that
\begin{equation*}
   (\langle u_1^2,[M]\rangle,\dots,\langle u_4^2,[M]\rangle)=(1,1,2,2)
\end{equation*}
up to ordering.

Therefore it follows that \(M\) and \(\C P^2\#\C P^2\) have, after reordering the facets of \(P\), the same \(BT^\infty\)-characteristic numbers.
Because these numbers are the only bordism invariants which are used in the proof of Theorem~\ref{sec:bordism-1}, it follows that \(M\) is weakly \(T\)-equivariantly homeomorphic to \(\C P^2\#\C P^2\). 
\end{proof}

\begin{remark}
  Corollary \ref{sec:bordism-2} also follows from the classification given in \cite{0216.20202} and the results of \cite{0486.57016}.
\end{remark}

\section{GKM-Graphs}
\label{sec:gkm}
Now we introduce the notion of a \emph{GKM-graph} of a quasitoric manifold following \cite{1138.53315}.

Let \(M^{2n}\) be a quasitoric manifold and \(M^{1}=\{x\in M;\dim Tx =1\}\).
Then \(M^T\) consists of isolated points and \(M^{1}\) has dimension two.

Let also
\begin{align*}
  V&=\{p_1,\dots,p_e\}=M^T\\
  E&=\{e_1,\dots,e_N\}=\{\text{components of } M^{1}\}
\end{align*}
and for \(i=1,\dots,N\) let \(\bar{e}_i\) be the closure of \(e_i\) in \(M\).
Then we have:
\begin{enumerate}
\item \(\bar{e_i}\) is an equivariantly embedded copy of \(\C P^1\).
\item \(\bar{e_i}-e_i\) consists of two points out of \(V\).
\item for \(p\in V\) we have \(\#\{e_i; p\in \bar{e}_i\}=n\).
\end{enumerate}

Therefore \(V\) and \(E\) are the vertices and edges of a graph \(\Gamma_M\).

We get a labeling of the edges of \(\Gamma_M\) by elements of the weight lattice of \(T\) as follows:
Let \(p,q\in V \cap \bar{e}_i\) then the weights \(\alpha_p,\alpha_q\) of \(T_p\bar{e}_i,T_q\bar{e}_i\) coincide up to sign and we define
\begin{equation*}
  \alpha:e_i\mapsto \alpha_p.
\end{equation*}
Then \(\alpha\) is determined up to sign and is called the \emph{axial function}
on \(\Gamma_M\).

We call \(\Gamma_M\) together with the axial function \(\alpha\) the GKM-graph of \(M\).

Now let \(M\) be a quasitoric manifold over the polytope \(P\).
Let \(\Gamma_P\) be the graph which consists of the edges and vertices of \(P\).
Then we have
\begin{equation*}
  \Gamma_M=\Gamma_P.
\end{equation*}

\begin{theorem}
  Let \(M\) be a quasitoric manifold. Then \(M\) is determined up to equivariant homeomorphism by \((\Gamma_M,\alpha)\).
\end{theorem}
\begin{proof}
  At first we introduce some notation.
  For a Lie-group \(G\) we denote its identity component of by \(G^0\).

  By \cite[p. 287,296]{0634.52005} the combinatorial type of \(P\) is uniquely determined by \(\Gamma_M\).
So we have to show that the function \(\lambda\) is determined by \(\alpha\).

Let \(F\) be a facet of \(P\) then we define
\begin{equation*}
  \lambda'(F)=\left(\bigcap_{e\subset F; e \text{ edge of  }P} \ker \chi^{\alpha(e)}\right)^0,
\end{equation*}
where \(\chi^{\alpha(e)}\) denotes the one-dimensional \(T\)-representation with weight \(\alpha(e)\).
We claim that \(\lambda'(F)=\lambda(F)\).
It follows immediately from the definition of \(\lambda\) that \(\lambda(F)\subset \lambda'(F)\).
Therefore we have to show that \(\lambda'(F)\) is at most one-dimensional.

Let \(x\in \pi^{-1}(F)^T\). Then we have
\begin{align*}
  T_x\pi^{-1}(F)&= \bigoplus_{\pi(x)\in e; e\subset F} \chi^{\alpha(e)}\\
  N_x(\pi^{-1}(F),M)&=\bigoplus_{\pi(x)\in e; e\not\subset F} \chi^{\alpha(e)}
\end{align*}
Therefore we have
\begin{equation*}
  \ker T_x\pi^{-1}(F) = \bigcap_{\pi(x)\in e; e\subset F} \ker\chi^{\alpha(e)}
\end{equation*}
But if
\begin{equation*}
  \dim\ker T_x\pi^{-1}(F) \geq 2
\end{equation*}
then the intersection
\begin{equation*}
  \ker T_x\pi^{-1}(F)\cap\ker N_x(\pi^{-1}(F),M)
\end{equation*}
is at least one-dimensional.
This contradicts with the effectiveness of the torus-action on \(M\).
\end{proof}

\bibliography{diss}{}
\bibliographystyle{amsplain}
\end{document}